\documentclass[envcountsame,envcountsect]{svmult}

\usepackage{amsmath}
\usepackage{amssymb}
\usepackage{latexsym}
\usepackage{amsxtra}
\usepackage{enumerate}
\usepackage{hyperref}
\usepackage[all]{xy}

\usepackage{mathptmx}       
\usepackage{helvet}         
\usepackage{courier}        
\usepackage{type1cm}        
\usepackage{makeidx}         
\usepackage{graphicx}        
\usepackage{multicol}        
\usepackage[bottom]{footmisc}
\usepackage{url}

\makeindex             
\usepackage[usenames, dvipsnames]{color}
\usepackage[utf8]{inputenc}

\smartqed

\newcommand{\NN}{\mathbb{N}}

\newcommand{\QQ}{\mathbb{Q}}
\newcommand{\FF}{\mathbb{F}}

\newcommand{\CC}{\mathbb{C}}
\newcommand{\PP}{\mathbb{P}}

\newcommand{\Gal}{\mathop{\rm Gal}\nolimits}

\newcommand{\p}{\mathfrak{p}}
\newcommand{\OO}{\mathcal{O}}
\newcommand{\Spec}{\mathop{\rm Spec}}

\newcommand{\Aut}{\mathop{\rm Aut}\nolimits}

\newcommand{\iso}{\stackrel{\sim}{\to}}

\newcommand{\gen}[1]{\mathopen\langle#1\mathclose\rangle}

\newcommand{\ord}{{\rm ord}}

\newcommand{\X}{\mathcal X}
\newcommand{\Y}{\mathcal Y}

\newcommand{\Xb}{\bar{X}}
\newcommand{\Yb}{\bar{Y}}

\newcommand{\xb}{\bar{x}}
\newcommand{\yb}{\bar{y}}

\newcommand{\zb}{\bar{z}}
\newcommand{\et}{{\rm et}}

\renewcommand{\D}{\mathcal{D}}
\newcommand{\Db}{\bar{D}}
\newcommand{\phib}{\bar{\phi}}

\begin{document}

\title*{Picard curves with small conductor}

\author{Michel B\"orner, Irene I.~Bouw, and Stefan Wewers}
\authorrunning{B\"orner, Bouw, Wewers}
\institute{Institut f\"ur Reine Mathematik, Universit\"at Ulm\\
\email{irene.bouw@uni-ulm.de},
\email{stefan.wewers@uni-ulm.de}}

\maketitle

\abstract{We study the conductor of Picard curves over $\QQ$, which is a
  product of local factors. Our results are based on previous results on
  stable reduction of superelliptic curves that allow to compute the conductor
  exponent $f_p$ at the primes $p$ of bad reduction. A careful analysis of the
  possibilities of the stable reduction at $p$ yields restrictions on the
  conductor exponent $f_p$. We prove that Picard curves over $\QQ$ always have
  bad reduction at $p=3$, with $f_3\geq 4$. As an application we discuss the
  question of finding Picard curves with small conductor. } 

\section{Introduction}
\label{sec:intro_picard}

Let $Y$ be a smooth projective curve of genus $g$ over a number field $K$. To
simplify the exposition, let us assume that $K=\QQ$. With $Y$ we can associate
an $L$-function $L(Y,s)$ and a conductor $N_Y\in\NN$. Conjecturally, the
$L$-function satisfies a functional equation of the form
\[
    \Lambda(Y,s) = \pm \Lambda(Y,2-s),
\]
where
\[
   \Lambda(Y,s) := \sqrt{N_Y}^s\cdot(2\pi)^{-gs}\cdot\Gamma(s)^g \cdot L(Y,s).
\]
By definition, both $L(Y,s)$ and $N_Y$ are a product of local factors. In
this paper we are really only concerned with the conductor, which can be
written as
\[
       N_Y = \prod_p p^{f_p}.
\]
The exponent $f_p$ is called the {\em conductor exponent} of $Y$ at $p$. It is
known that $f_p$ only depends on the ramification of the local Galois
representation associated with $Y$. In particular, if $Y$ has good reduction
at $p$ then $f_p=0$. If $Y$ has bad reduction at $p$ then the computation of
$f_p$ can be quite difficult. Until recently, an effective method for
computing $f_p$ was only known for elliptic curves (\cite{SilvermanAT}, \S
IV.10) and for genus $2$ curves if $p\neq 2$ (\cite{Liu94}).

It was shown in \cite{superell} that $f_p$ can effectively
be computed from the {\em stable reduction} of $Y$ at $p$. Moreover, for
certain families of curves (the {\em superelliptic curves}) we gave a rather
simple recipe for computing the stable reduction. The latter result needed the
assumption that $p$ does not divide the degree $n$. In \cite{superp} this
restriction is removed for  superelliptic curves of prime degree. 

In the present paper we systematically study the case of {\em Picard
  curves}. These are superelliptic curves of genus $3$ and degree $3$, given
by an equation of the form
\[
    Y:\; y^3 = f(x) = x^4+a_3x^3+a_2x^2+a_1x+a_0,
\]
with $f\in\QQ[x]$ separable. Picard curves form in some sense the next
family of curves to study after hyperelliptic curves. They are
interesting for many reasons and have been intensively studied, see
e.g. \cite{Picard}, \cite{Holzapfel}, \cite{KoikeWeng}, and
\cite{Quine}.  

Our main results classify all possible configurations for the stable
reduction of a Picard curve at a prime $p$, and use this to determine
restrictions on the conductor exponents.  For
instance, we prove the following.

\begin{theorem} 
Let $Y$ be a Picard curve over $\QQ$.
\begin{enumerate}[(a)] 
\item Then $Y$ has bad reduction
  at $p=3$, and $f_3\geq 4$.
\item For $p=2$ we have $f_2\neq 1$.
\item  For $p\geq 5$ we have $f_p\in \{0,2,4,6\}$.
\end{enumerate}
\end{theorem}

Theorem \ref{thm:main} is a somewhat stronger version of the first
statement. Theorem \ref{thm:tame2} contains the last two
statements. We also give explicit examples, showing that at least part
of our results are sharp.  Our result can be seen as a complement, for
Picard curves, to a result of Brumer--Kramer (\cite{BrumerKramer},
Theorem 6.2), who prove an upper bound for $f_p$ for abelian varieties
of fixed dimension. Since the conductor of a curve coincides with that
of its Jacobian, the result applies to our situation, as well. A more
careful case-by-case analysis, combined with ideas from
\cite{BrumerKramer}, could probably be used to obtain a more precise
list of possible values for the conductor exponent at $p=2,3$, as well.

In the last section we discuss the problem of constructing Picard
curves with small conductor. As a consequence of the Shafarevich
conjecture (aka Faltings' Theorem), there are at most a finite number
of nonisomorphic curves of given genus and of bounded conductor. But
except in very special cases, no effective proof of this theorem is
known.

In his recent PhD thesis, the first named author has made an extensive search
for Picard curves with good reduction outside a small set of small primes, and
computed their conductor. The Picard curve with the smallest conductor that was
found is the curve
\[
    Y:\;y^3 =x^4-1,
\]
which has conductor
\[
     N_Y = 2^63^6 = 46656.
\]
We propose as a subject for further research to either prove that the above
example is the Picard curve over $\QQ$ with the smallest possible conductor,
or to find (one or all) counterexamples. We believe that the methods presented
in this paper may be very helpful to achieve this goal. 

\section{Semistable reduction}
\label{sec:semistable}

We first introduce the general setup concerning the stable reduction and the
conductor exponents of Picard curves.  As explained in the introduction, the
conductor exponent is a local invariant, encoding information about the
ramification of the local Galois representation associated with the
curve. Therefore, we may replace the number field $K$ by its strict
henselization. In other words, we may work from the start over a henselian
field of mixed characteristic with algebraically closed residue field.

\subsection{Setup and notation} \label{setup}

Throughout Section \ref{sec:semistable} - \ref{sec:tame} the letter $K$ will
denote a field of characteristic zero that is henselian with respect to a
discrete valuation. We denote the valuation ring by $\OO_K$, the maximal
ideal of $\OO_K$ by $\p$ and the residue field by $k=\OO_K/\p$. We assume that
$k$ is algebraically closed of characteristic $p>0$. The most important
example for us is when $K=\QQ_p^{\rm nr}$ is the maximally unramified
extension of the $p$-adic numbers. Then $\p=(p)$ and $k=\bar{\FF}_p$.  

Let $Y/K$ be a Picard curve, given by the equation
\begin{equation} \label{eq:picard_wild1}
  Y:\; y^3 = f(x),
\end{equation}
where $f\in K[x]$ is a separable polynomial of degree $4$. We set
$X:=\PP^1_K$ and interpret \eqref{eq:picard_wild1} as a finite cover
$\phi:Y\to X,\;(x,y)\mapsto x,$ of degree $3$. 

By the Semistable Reduction Theorem (see \cite{DeligneMumford69}),
there exists a finite extension $L/K$ such that the curve
$Y_L:=Y\otimes_K L$ has semistable reduction. Since $g(Y)=3\geq 2$,
there even exists a (unique) distinguished semistable model
$\Y\to\Spec\OO_L$ of $Y_L$, the {\em stable model}
(\cite{DeligneMumford69}, Corollary 2.7). The special fiber
$\Yb:=\Y_s$ of $\Y$ is called the {\em stable reduction} of $Y$. It is
a stable curve over $k$ (\cite{DeligneMumford69}, \S~1), and it only
depends on $Y$, up to unique isomorphism.

It is no restriction to assume that the extension $L/K$ is Galois and contains
a third root of unity $\zeta_3\in L$. Then the cover $\phi_L:Y_L\to X_L$ (the
base change of $\phi$ to $L$) is a Galois cover. Its Galois group $G$ is
cyclic of order $3$, generated by the element $\sigma$ which is determined by
\[
       \sigma(y)=\zeta_3 y.
\]
Let $\Gamma:=\Gal(L/K)$ denote the Galois group of the extension $L/K$. 
The group $\Gamma$ acts faithfully and in a natural way on the scheme
$Y_L=Y\otimes_K L$. We denote by $\tilde{G}$ the subgroup of $\Aut(Y_L)$
generated by $G$ and the image of $\Gamma$. By definition, $\tilde{G}$ is a
semidirect product,
\[
      \tilde{G} = G\rtimes\Gamma.
\]
The action of $\Gamma$ on $G$ via conjugation is determined by the following
formula: for $\tau$ in $\Gamma$ we have 
\begin{equation} \label{eq:picard_wild2}
     \tau\sigma\tau^{-1} = \begin{cases}
          \sigma & \text{if $\tau(\zeta_3)=\zeta_3$,}\\
          \sigma^2 & \text{if $\tau(\zeta_3)=\zeta_3^2$.}
                           \end{cases}
\end{equation}

Because of the uniqueness properties of the stable model, the action of
$\tilde{G}$ on $Y_L$ extends to an action on $\Y$. By restriction, we see that
$\tilde{G}$ has a natural, $k$-linear action on $\Yb$. This action will play a
decisive role in our analysis of the stable reduction $\Yb$. For the rest of
this subsection we focus on the action of the subgroup
$G\subset\tilde{G}$. The role of the subgroup $\Gamma\subset\tilde{G}$ will
become important later. 

\begin{remark} \label{rem:ssmodels}
\begin{enumerate}[(a)]
\item The quotient scheme $\X:=\Y/G$ is a semistable model of
  $X_L=\PP^1_L$, see e.g.~\cite{Raynaud99}, Cor.~1.3.3.i. Since the
  map $\Y\to\X$ is finite and $\Y$ is normal, $\Y$ is the
  normalization of $\X$ in the function field of $Y_L$. This means
  that $\Y$ is uniquely determined by a suitable semistable model $\X$
  of $X_L$.
\item
  Let $\Xb:=\X\otimes k$ denote the special fiber of $\X$ and
  $\bar{\phi}:\Yb\to\Xb$ the induced map. We note that $\bar{\phi}$ is a
  finite $G$-invariant map. It is not true in general that
  $\Yb/G=\Xb$. However, the natural map $\Yb/G\to\Xb$ is radicial and in
  particular a homeomorphism (see e.g.~\cite{Raynaud99}, Prop.~2.4.11). 
\item
  Every irreducible component $W\subset\Yb$ is smooth. To see this
  note that the quotient of $W$ by its stabilizer in $G$ is
  homeomorphic to an irreducible component $Z\subset\Xb$, which is a
  smooth curve of genus $0$. If $W$ has a singular point, then
  $\sigma$ acts on $W$ and permutes the two branches of
  $W$ passing through this point. But since $\sigma$ has order $3$,
  this is impossible. 
\end{enumerate}
\end{remark}
Let $\Delta_{\Yb}$ denote the component graph of $\Yb$: the vertices
are the irreducible components of $\Yb$ and the edges correspond to
the singular points. The stability condition for $\Yb$ means that an
irreducible component of genus $0$ corresponds to a vertex of
$\Delta_{\Yb}$ of degree $\geq 3$. The number of loops of
$\Delta_{\Yb}$ is given by the well known formula
\begin{equation} \label{eq:gamma1}
     \gamma(\Yb):=\dim_\QQ H^1(\Delta_{\Yb},\QQ) = r-s+1,
\end{equation}
where $r$ is the number of edges and $s$ the number vertices of
$\Delta_{\Yb}$.

The curve $\Xb$ is also semistable, but in general not stable. Since $\Xb$ has
arithmetic genus $0$, the component graph $\Delta_{\Xb}$ is a tree, and every
vertex corresponds to a smooth curve of genus $0$. 
It follows from Remark \ref{rem:ssmodels} that
$\Delta_{\Xb}=\Delta_{\Yb}/G$. 

\begin{lemma} \label{lem:stable1}
  If $W\subset\Yb$ is an irreducible component, then $\sigma(W)=W$.
\end{lemma}

\begin{proof} To derive a contradiction, we
 assume that $W_1,W_2,W_3\subset\Yb$ are
three distinct components that form a single $G$-orbit. Then $W_i\iso
Z:=\bar{\phi}(W_i)$. Since $Z$ is a component of $\bar{X}$, we
conclude that $g(W_i)=0$, for $i=1,2,3$. The stability condition on
$\Yb$ implies that each $W_i$ contains at least three singular points
of $\Yb$. Hence $Z$ also contains at least three singular points of
$\Xb$.

Let $\Yb\to\Yb_0$ denote the unique morphism which contracts all components of
$\Yb$ except the $W_i$ and which is an isomorphism on the intersection of
$\cup_i W_i$ with the smooth locus of $\Yb$. Similarly, let $\Xb\to \Xb_0$ be
the map contracting all components of $\Xb$ except $Z$. These maps fit into a
commutative diagram
\[
  \xymatrix{ \Yb \ar[r]\ar[d] & \Yb_0 \ar[d] \\
             \Xb \ar[r]       & \Xb_0,       \\
           }
\]
where the vertical arrows are quotient maps by the group $G$ (at least
for the underlying topological spaces). Also, $\Xb_0\cong Z$. 

Let $\xb\in Z$ be one of the singular point of $\Xb$ lying on $Z$, and
let $T\subset\Xb$ the closed subset which is contracted to $\xb\in
Z=\Xb_0$. Then $T$ is a nonempty and connected union of irreducible
components of $\Xb$ and hence a semistable curve of genus $0$. In
particular, the component graph of $T$ is a tree. Let $Z'\subset T$ be
a tail component. As a component of $\Xb$, $Z'$ intersects the rest of
$\Xb$ in at most two points. Let $W'\subset\Yb$ be an irreducible
component lying above $Z'$. The stability of $\Yb$ implies that
$\sigma(W')=W'$ and that the action of $\sigma$ on $W'$ is
nontrivial. (Otherwise $W'$ would be homeomorphic to $Z'$, and
hence $W'$ would be a component of genus $0$ intersecting the rest of
$\Yb$ in at most two points. )
It follows  that the inverse image $S\subset\Yb$ of $T$ is connected.
Note that $S$ meets the component $W_i$ in the unique point on $W_i$
above $\xb$. Since $S$ is connected, it follows that the map
$\Yb\to\Yb_0$ contracts $S$ to a single point. 

We conclude that the curve $\Yb_0$ has at least three distinct
singular points where all three components $W_i$ meet. Equation
(\ref{eq:gamma1}) implies that $\gamma(\Yb_0)$ is at least $1$.  It
follows that the arithmetic genus of $\Yb_0$ is $\geq 4$, and hence
$g(\Yb)\geq 4$ as well. This is a contradiction, and the lemma follows.
\qed
\end{proof}

\subsection{The conductor exponent}

Let $\mathfrak{c}_\p$ be the {\em conductor} of the
$\Gal(\bar{K}/K)$-representation $H^1_\et(Y_{\bar{K}},\QQ_\ell)$, see
\cite{SerreZeta}. By definition, this is an ideal of $\OO_K$ of the form
\[
      \mathfrak{c}_\p = \p^{f_\p},
\]
with $f_\p\geq 0$. The integer $f_\p$ is called the {\em conductor exponent}
of $Y/K$.\footnote{When working in a local context, $f_\p$ is often simply
  called the conductor of $Y$.}

We recall from \cite{superell} an explicit formula for $f_\p$,
in terms of the action of $\Gamma=\Gal(L/K)$ on
$\Yb$. For this we let $\Gamma^u\subset\Gamma$, for $u\geq 0$, denote
the $u$th higher ramification group (in the { upper
  numbering}). We set $\Yb^u:=\Yb/\Gamma^u$. Note that $\Yb^u$ is a
semistable curve for all $u$. Note also that $\Gamma=\Gamma^0$ because the
residue field $k$ is assumed to be algebraically closed. 

\begin{proposition}\label{prop:fp}
  The conductor exponent of the curve $Y/K$ is given by 
  \begin{equation} \label{eq:fp1}
     f_\p = \epsilon + \delta,
  \end{equation}
  where 
  \begin{equation} \label{eq:fp2}
     \epsilon := 6 - \dim\, H^1_{\rm et}(\Yb^0,\QQ_\ell)
  \end{equation}
  and
  \begin{equation} \label{eq:fp3}
     \delta := \int_{0}^\infty \big(6-2g(\Yb^u)\big) {\rm d}u.
  \end{equation}
\end{proposition}        

\begin{proof}
See \cite{superell}, Theorem 2.9 and \cite{ICERM}, Corollary 2.14.
\qed
\end{proof}
 
The \'etale cohomology group $H^1_{\rm
  et}(\Yb^u,\QQ_\ell)$ decomposes as
\[
H^1_{\rm et}(\Yb^u,\QQ_\ell)=\oplus_{W} H^1_{\rm
  et}(W,\QQ_\ell)\oplus H^1(\Delta_{\Yb^u},\QQ_\ell),
\]
where the first sum runs over the set of irreducible components $W$ of
the normalization of $\Yb^u$ and $\Delta_{\Yb^u}$ is the graph of
components of $\Yb^u$. (See \cite{superell}, Lemma 2.7.(1).)
Therefore, the second term in \eqref{eq:fp2} can be written as
\begin{equation}
   \dim\, H^1_{\rm et}(\Yb^0,\QQ_\ell) = \sum_W \dim\, H^1_{\rm
  et}(W,\QQ_\ell) \;+\; \dim\, H^1(\Delta_{\Yb^0}).
\end{equation}

The arithmetic genus of
$\Yb^u$, which occurs in \eqref{eq:fp3}, is given by the formula
\begin{equation}\label{eq:H11}
  g(\Yb^u)=\sum_{W} g(W) \;+\; \dim\,  H^1(\Delta_{\Yb^u}).
\end{equation}
For future reference we note that $\dim\, H^1_{\rm
  et}(W,\QQ_\ell)=2g(W)$. The integer $\gamma(\Yb^0):=\dim\,
H^1(\Delta_{\Yb^0})$ can be interpreted as the number of loops of the
graph $\Delta_{\Yb^0}$. It is bounded by $g(\Yb^0)$, and hence by
$g(Y)=3$.

\begin{lemma} \label{lem:delta}
  The following statements are equivalent.
  \begin{enumerate}[(a)]
  \item
    $\delta=0$.
  \item
    $\Gamma^u$ acts trivially on $\Yb$, for all $u>0$.
  \item
    The curve $Y$ has semistable reduction over a tamely ramified extension of
    $K$. 
  \end{enumerate}
\end{lemma}

\begin{proof}  Assume that
  $\delta=0$. By \eqref{eq:fp3} this means that $3=g(\Yb)=g(\Yb^u)$ for all
  $u>0$. Using \eqref{eq:H11} one easily shows that this means that $\Gamma^u$
  acts trivially on the component graph $\Delta_{\Yb}$ of $\Yb$. Moreover, for
  every component $W\subset\Yb$ we have $g(W)=g(W/\Gamma^u)$. It follows that
  $\Gamma ^u$ acts trivially on $\Yb$. We have proved the implication
  (a)$\Rightarrow$(b). The implication (b)$\Rightarrow$(c) follows from
  \cite{liu}, Theorem 4.44. The implication (c)$\Rightarrow$(a) follows
  immediately from the definition of $\delta$.  \qed
\end{proof}


\section{The wild case: $p=3$}
\label{sec:picard_wild}

In this section we assume that $p=3$. We first analyze the special fiber of
the stable model of $Y_L$, and show that there are essentially five reduction
types. From \S~\ref{sec:picard_wild2} we consider the case where $K$ is
absolutely unramified, and derive a lower bound for the conductor exponent
$f_3$.

\subsection{The stable model}
\label{sec:picard_wild1}

We keep all the notation introduced in \S~\ref{sec:semistable}. In addition,
we assume that $p=3$. Lemma \ref{lem:stable1} implies that we can distinguish
between two types of irreducible components of $\Yb$.

\begin{definition} \label{def:etale}
  An irreducible component $W\subset \Yb$ is called {\em \'etale} if the
  restriction $\sigma|_W\in\Aut_k(W)$ is nontrivial. If $\sigma|_W$ is the
  identity, then $W$ is called an {\em inseparable component}.
\end{definition}
 
Let $W\subset\Yb$ be an irreducible component, and let
$Z:=\bar{\phi}(W)\subset\Xb$ be its image. Then $Z$ is an irreducible
component of $\Xb$ and hence a smooth curve of genus $0$. Lemma
\ref{lem:stable1} shows that $\sigma(W)=W$. It follows that $W/G\to Z$ is a
homeomorphism.  If $W$ is an inseparable component, then $W\to Z$ is a purely
inseparable homeomorphism (since $W\to Z$ has degree $3$, this can only happen
when $p=3$). It follows that every inseparable component has genus zero.

If $W$ is an \'etale component, then $Z\cong W/G$, and $W\to Z$ is a
$G$-Galois cover. For future reference we recall that the Riemann--Hurwitz
formula for wildly ramified Galois covers of curves yields
\begin{equation}\label{eq:RH}
2g(W)-2=-2\cdot 3+\sum_{z} 2(h_z+1),
\end{equation}
where the sum runs over the branch points of $W\to Z$ and $h_z$ is the
(unique) jump in the filtration of the higher ramification groups in
the lower numbering. We have that $h_z\geq 1$ is prime to $p$
(\cite{SerreCL}, \S~IV.2, Cor.~2 to Prop.~9).

\begin{theorem} \label{thm:wild1}
  We are in exactly one of the following five cases.
  \begin{enumerate}[(a)]
  \item
    The curve $\Yb$ is smooth and irreducible.
  \item There are exactly two components $W_1,W_2$ which are both \'etale,
    meet in a single point, and have genus $g(W_1)=2$, $g(W_2)=1$.
  \item
    There are three \'etale components $W_1,W_2,W_3$ of genus one, and one
    inseparable component $W_0$ of genus zero. For $i=1,2,3$, $W_i$ intersects
    $W_0$ in a unique point, and these intersection points are precisely the
    singular points of $\Yb$.
  \item
    There are two components $W_1,W_2$ which are \'etale of genus $g(W_1)=1$,
    $g(W_2)=0$. There are exactly three singular points, which form an orbit
    under the action of $G$, and where $W_1$ and $W_2$ meet.
  \item
    There are three components $W_1,W_2,W_3$, which are \'etale and of genus
    $g(W_1)=g(W_2)=0$ and $g(W_3)=1$, and four singular points. Three of the
    singular points are points of intersection of $W_1$ and $W_2$, and form an
    orbit under the action of $G$. The fourth singular point is the point of
    intersection of $W_2$ and $W_3$. 
  \end{enumerate}
\end{theorem}

\begin{proof}   Let $r_1$
  (resp. $s_1$) be the number of singular points (resp.\ irreducible
  components) of $\Yb$ which are fixed by $\sigma$, and let $r_2$
  (resp. $s_2$) be the number of orbits of singular point
  (resp.\ irreducible components) of $\Yb$ of length $3$. Lemma
  \ref{lem:stable1} states that $s_2=0$. Therefore, \eqref{eq:gamma1}
  becomes
\begin{equation} \label{eq:gamma2}
  \gamma(\Yb) = r-s+1 = r_1+3r_2 - s +1.
\end{equation}
Because $\Delta_{\Xb}=\Delta_{\Yb}/G$ is a tree, we have
\begin{equation} \label{eq:gamma3} 
  \gamma(\Xb) :=\dim H^1(\Delta_{\Xb}) = r_1+r_2 -s +1 = 0.
\end{equation}
Combining \eqref{eq:gamma2} and \eqref{eq:gamma3} we obtain
\begin{equation} \label{eq:gamma4}
  \gamma(\Yb) = 2r_2.
\end{equation}
Since $0\leq \gamma(\Yb)\leq 3$, we conclude that $\gamma(\Yb)\in\{0,2\}$ and
$r_2\in\{0,1\}$.  

\bigskip
{\bf Case 1}: $r_2=0$ and $\gamma(\Yb)=0$.\\ In this case
$\Delta_{\Yb}$ is a tree, and the sum of the genera of all irreducible
components is $3$. In particular, there are at most $3$ components of
genus $>0$. Moreover, the stability condition implies that every
component of genus zero contains at least three singular points of
$\Yb$. It is an easy combinatorial exercise to see that this leaves us
with exactly four possibilities for the tree $\Delta_{\Yb}$. Going
through these four cases we will see that one of them is excluded,
while the remaining three correspond to Case (a), (b), and (c) of
Theorem \ref{thm:wild1}.

The first case is when $\Yb$ has a unique irreducible component. Then $\Yb$ is
smooth. This is Case (a) of the lemma. Secondly, there may be two irreducible
components, of genus $1$ and $2$, and a unique singular point. This
corresponds to Case (b). 

Thirdly, there may be three irreducible components, each of genus $1$,
and two singular points. We claim that this case cannot occur. Indeed,
one of the three components would contain two singular points, and
each of these two points must be a fixed point of $\sigma$. It follows
that the $G$-cover $W\to Z=W/G$ is ramified in at least two
points. The Riemann--Hurwitz formula (\ref{eq:RH}) implies that
$g(W)\geq 2$. This yields a contradiction, and we conclude that this
case does not occur.

 Finally, in the last case, there are four singular points and four
 irreducible components. Three of them have genus $1$ and one has
 genus zero. The component of genus zero necessarily contains all three
 singular points. A similar argument as in the previous case
 shows that the genus-$0$ component cannot be \'etale. This corresponds to
 Case (c).

\bigskip
{\bf Case 2}: $\gamma(\Yb)=2$ and $r_2=1$.\\ In this case the sum of
the genera of all components is equal to $1$. Therefore, there must be
a unique component of genus $1$, and all other components have genus
$0$. Let $W_1$ and $W_2$ be two components which meet in a singular
point $\yb$ such that $\sigma(\yb)\neq \yb$. Since $\sigma(W_i)=W_i$
for $i=1,2$ (Lemma \ref{lem:stable1}), $W_1$ and $W_2$ are \'etale
components and intersect each other in exactly three points (the
$G$-orbit of $\yb$).

If there are no further components, we are in Case (d). 
Assume that there exists a third component $W_3$. Let $T\subset \Yb$
be the maximal connected union of components which contains $W_3$ but
neither $W_1$ nor $W_2$. Then $T$ contains a unique component $W_0$
which meets either $W_1$ or $W_2$ in a singular point. The component
graph of $T$ is a tree, and we consider $W_0$ as its root. By the
stability condition, every tail component of $T$ must have positive
genus, so $T$ has a unique tail. If $W_0$ is not this tail, it has
genus $0$ and intersects the rest of $\Yb$ in exactly $2$ points.  This
contradicts the stability condition. We conclude that $\Yb$ has exactly three
components, of genus $g(W_1)=g(W_2)=0$ and $g(W_3)=1$. This is Case (e) of the
lemma. Now the proof is complete. 
\qed
\end{proof}

\subsection{A lower bound for $f_3$}\label{sec:picard_wild2}

We continue with the assumptions from the previous subsection. In addition, we
assume that $K$ is absolutely unramified. By this we mean that $\p=(3)$. Under
this assumption, we prove a lower bound for the conductor exponent
$f_3:=f_\p$. In fact, we will give a lower bound for $\epsilon$, where
$f_3=\epsilon+\delta$ is the decomposition from Proposition \ref{prop:fp}. If
$L/K$ is at most tamely ramified, then $\delta=0$ (Lemma \ref{lem:delta}). In
this case, our bounds are sharp.



Since $K$ is absolutely unramified, the third root of unity $\zeta_3\in L$ is
{\em not} contained in $K$. Therefore, there exists an element $\tau\in
\Gamma=\Gal(L/K)$ such that $\tau(\zeta_3)=\zeta_3^2$. Let $m$ be the order of
$\tau$. After replacing $\tau$ by a suitable odd power of itself we may assume
that $m$ is a power of $2$. We keep this notation fixed for the rest of this
paper. Recall that the semidirect product $\tilde{G}=G\rtimes \Gamma$ acts on
$\Yb$ in a natural way.

The following observation is crucial for our analysis of the conductor
exponent. 

\begin{lemma}\label{lem:zeta3}
  Let $W\subset\Yb$ be an \'etale component such that $\tau(W)=W$. Then inside
  the automorphism group of $W$ we have
\begin{equation}\label{eq:zeta3}
\tau\circ \sigma\circ\tau^{-1}=\sigma^2\neq \sigma.
\end{equation}
  In particular, $\tau|_W$ is nontrivial. 
\end{lemma}

\begin{proof}
The statement follows immediately from Equation \eqref{eq:picard_wild2} and
Definition \ref{def:etale}.
\qed
\end{proof}

Despite its simplicity, Lemma \ref{lem:zeta3} has the following
striking consequence. Note that we consider potentially good but not
good reduction as bad reduction in this paper.

\begin{proposition}\label{prop:badat3}
  Assume that $\p=(3)$. Then every Picard curve $Y$ over $K$ has bad
  reduction.
\end{proposition} 

\begin{proof}
Lemma \ref{lem:zeta3} implies that $Y$ acquires semistable reduction
only after passing to a ramified extension $L\ni\zeta_3$. Therefore
$Y/K$ does not have good reduction. The fact that $f_3\neq 0$ follows
from Proposition \ref{prop:fp}, together
with the fact that $\tau$ acts nontrivially on each irreducible
component of $\Yb$ (Lemma \ref{lem:zeta3}).
\qed
\end{proof}

In order to prove more precise lower
bounds for $f_3$, we need to analyze the action of $\sigma$ and $\tau$
on $\Yb$ in more detail.

\begin{lemma}\label{lem:genuscomp}
Let $W\subset\Yb$ be an \'etale component. Then one of the following
cases occurs:
\[
\begin{array}{c|c|c|c}
g(W) &\,r\,&\,h\,& g(W/\Gamma^0)\\
\hline
0&1&1& 0\\
1&1&2& 0\\
2&2& (1,1)& 1\\
3&1&4& 0
\end{array}\]
Here  $r$ is the number of ramification points of the
$G$-cover $W\to Z:=W/G$ and $h$ lists the set of lower jumps. The
fourth column gives an upper bound for the genus of $W/\Gamma^0$.
\end{lemma}

\begin{proof}
Recall that we have assumed that the order $m$ of $\tau$ is a power of
$2$. 

The Riemann--Hurwitz formula (\ref{eq:RH}) immediately yields the
cases for $g(W), r$, and $h$ stated in the lemma, together with one
additional possibility: the curve $W$ has genus $3$ and $\phi:W\to
Z\cong \PP^1$ is branched at two points, with lower jump $1$ and $2$,
respectively.  We claim that this case does not occur.

Assume that $W$ is an \'etale component of $\Yb$ such that $\phi:W\to
Z$ is branched at $2$ points. Lemma \ref{lem:zeta3} implies that
$\tau$ acts nontrivially on $W$.  Since $\tau$ normalizes $\sigma$ and
the two ramification points have different lower jumps, it follows
that $\tau$ fixes both ramification points $w_i$ of $\phi$.  We
conclude that $H:=\langle \sigma, \tau\rangle$ acts on $W$ as a
nonabelian group of order $6$ fixing the $v_i$.

We write $h_i$ for lower jump of $w_i$.  Lemma 2.6 of \cite{pries}
implies that $\gcd(h_i, m)$ is the order of the prime-to-$3$ part of
the centralizer of $H$.  Since $\gcd(h_1, m)\neq \gcd(h_2, m)$ we
obtain a contradiction, and conclude that this case does not occur.

\bigskip
We compute an upper bound for the genus of $W/\langle \tau\rangle$ in
each of the remaining cases. This is also an upper bound for $g(W/\Gamma^0)$.

In the case that $g(W)=0$ there is nothing to prove.  In the case that
$g(W)=1$, the automorphism $\tau$ fixes the unique ramification point of $\phi$,
hence $g(W/\Gamma^0)=0.$

Assume that $g(W)=2$. The Riemann--Hurwitz formula immediately implies
that that $g(W/\langle \tau\rangle)\leq 1$.

Finally, we consider the case that $g(W)=3$, i.e.~$Y$ has potentially
good reduction. As before, we have that $\tau$ fixes the unique fixed
point of $\sigma$. Put $H=\langle \sigma, \tau\rangle$. Lemma
\ref{lem:zeta3} together with the assumption that the order $m$ of
$\tau$ is a power of $2$ implies that the order of the prime-to-$p$
centralizer of $H$ is $\gcd(h=4, m)=m/2$. It follows that $m=8$.
Since $\tau$ has at least one fixed point on $W$,
namely the point at $\infty$, the Riemann--Hurwitz formula implies that
$g(W/\langle \tau\rangle)=0$. 
This finishes the proof of the lemma.
\qed
\end{proof}





We have now all the necessary tools to prove our main theorem. 

\begin{theorem} \label{thm:main}
  Assume $\p=(3)$, and let $Y$ be a
  Picard curve over $K$. The conductor exponent $f_3$ of $Y/K$ satisfies
  \[
       f_3\geq 4.
  \]
  Moreover:
  \begin{enumerate}[(a)]
  \item
    If $f_3\leq 6$ then $Y$ achieves semistable reduction over a tamely
    ramified extension $L/K$. 
  \item If $f_3=4$ then we are  in Case (b) or Case (c)
    from Theorem \ref{thm:wild1}.
  \item
    If $f_3=5$ then we are  in Case (d) or in Case
    (e) of Theorem \ref{thm:wild1}.  
  \end{enumerate}
\end{theorem}

\begin{proof} 
We use the assumptions and notations from the beginning of
\S~\ref{sec:picard_wild2}. Recall that the inertia subgroup
$\Gamma^0\subset\Gamma:=\Gal(L/K)$ acts on the geometric special fiber
$\Yb$ of the stable model of $Y_L$ and that .the quotient
$\Yb^0=\Yb/\Gamma^0$ is again a semistable curve.

\bigskip
{\bf Claim}:  We have that
\begin{equation} \label{eq:main1}
   \dim\, H^1_\et(\Yb^0,\QQ_\ell) \leq 2.
\end{equation}
Note that \eqref{eq:main1}, together with \eqref{eq:fp1} and \eqref{eq:fp2},
immediately implies the first statement $f_3\geq 4$ of the theorem. 
 
Recall from (\ref{eq:H11}) and (\ref{eq:gamma1}) that the contribution
of a smooth component $W$ of $\Yb^0$ to $ \dim\,
H^1_\et(\Yb^0,\QQ_\ell)$ is $2g(W)$. The contribution of
$H^1(\Delta_{\Yb})$ to $ \dim\, H^1_\et(\Yb^0,\QQ_\ell)$ is
$\gamma(\Yb^0)$, which is less than or equal to $g(\Yb^0)$.

Let $W\subset\Yb$ be an irreducible component, and denote by $W^0\subset\Yb_0$
its image in $\Yb^0$. Clearly, $g(W^0)\leq g(W)$. Moreover, if $\tau(W)=W$
then Lemma \ref{lem:genuscomp} shows that $g(W^0)\leq 1$. 

Let us consider each case of Theorem \ref{thm:wild1} separately. In Case (a),
$\Yb$ is smooth and irreducible of genus $3$. Then $\Yb^0$ is also smooth
and irreducible, and Lemma \ref{lem:genuscomp} shows that $g(\Yb^0)=0$. So in
Case (a) we have proved $\dim\, H^1_\et(\Yb^0,\QQ_\ell) = 0$, which is
strictly stronger than \eqref{eq:main1}. Similarly, in Case (b) Lemma
\ref{lem:genuscomp} shows that $\Yb^0$ consists of two irreducible components
which meet in a single point. One of these components has genus zero, the
other one has genus $\leq 1$. Therefore, \eqref{eq:main1} holds in Case
(b). 
  
Assume that we are in Case (c). Let $W_1,W_2,W_3$ denote the three components
of genus $1$, and $W_i^0$, $i=1,2,3$, their images in $\Yb^0$. Since the order
of $\tau$ is a power of two, $\tau$ fixes exactly one of these
components (say $W_1$), or all three. In the first case, $g(W_1^0)=0$ by
Lemma \ref{lem:genuscomp}, and $W_2^0=W_3^0$. Therefore, $\dim\,
H^1_\et(\Yb^0,\QQ_\ell)=1$. In the second case, $g(W_i^0)=0$ for $i=1,2,3$,
and $\dim\, H^1_\et(\Yb^0,\QQ_\ell)=0$. In both cases, \eqref{eq:main1}
holds. 

Now assume that we are in Case (d). The action of $\Gamma^0$ must fix
both components $W_1,W_2$, since $g(W_1)\neq g(W_2)$. Lemma
\ref{lem:genuscomp} shows that $g(W_i/\Gamma^0)=0$, for $i=1,2$. Also,
$\tau$ permutes the three singular points of $\Yb$. But these points
form one orbit under the action of $G$. Hence it follows from
\eqref{eq:zeta3} that $\tau$ fixes exactly one singular point and
permutes the other two. We conclude that the curve $\Yb^0$ has two
smooth components of genus $0$ which meet in at most two points. We
conclude that $\dim\, H^1_\et(\Yb^0,\QQ_\ell)\leq 1$. A similar
analysis shows that the same conclusion holds in Case (e). This proves
the claim (\ref{eq:main1}).

While proving the claim, we have shown  the following stronger conclusion:
\begin{equation} \label{eq:main2}
  \dim\, H^1_\et(\Yb^0,\QQ_\ell) \in 
    \begin{cases}
       \{0\}, & \text{Case (a),} \\
       \{0,2\}, & \text{Case (b), (c),} \\
       \{0,1\}, & \text{Case (d), (e).}
   \end{cases}
\end{equation}
It follows that $\epsilon=6$ in Case (a), $\epsilon\in\{4,6\}$ in the
Cases (b) and (c), and $\epsilon\in\{5,6\}$ in the Cases (d) and (e). 

The remaining statement that $Y$ acquires stable reduction over a
tamely ramified extension $L$ of $K$ in the case that $f_3\leq 6$
follows from Lemma \ref{lem:delta}.
\qed
\end{proof}

\begin{corollary}
  If $\p=(3)$ and $Y$ has potentially good reduction, then $f_3\geq 6$.
\end{corollary}

\subsection{Examples} \label{subsec:wild_examples}

In this section we discuss two explicit examples of Picard curves over
$\QQ_3^{\rm nr}$ in some detail. These examples show, among other things, that
the lower bounds for $f_3$ given by Theorem \ref{thm:main} are sharp. 

Let us fix some notation. We set $K:=\QQ_3^{\rm nr}$.  Given a suitable finite
extension $L/K$, we denote by $v_L$ the unique extension of the $3$-adic
valuation to $L$ (which is normalized such that $v_L(3)=1$). We let $F(X_L)$
denote the function field of $X_L:=\PP^1_L$, and identify $F(X_L)$ with the
rational function field $L(x)$.  For a Picard curve $Y$ over $K$ given by
$y^3=f(x)$ for a quartic polynomial $f\in K[x]$ the function field $F(Y_L)$ of
$Y_L$ is the degree-$3$ extension of $F(X_L)$ obtained by adjoining the
function $y$.

Let $\X$ be a semistable model of $X_L$, and let
$Z_1,\ldots,Z_n\subset\Xb:=\X\otimes\FF_L$ denote the irreducible components
of the special fiber. Since each $Z_i$ is a prime divisor on $\X$, it gives
rise to a discrete valuation $v_i$ on $F(X_L)$, extending $v_L$. It has the
property that the residue field of $v_i$ can be naturally identified with the
function field of $Z_i$. Since $X_L$ is simply a projective line and $\X$ is a
semistable model, the valuations $v_i$ have a simple description, as
follows. For all $i$, there exists a coordinate $x_i\in F(X_L)$ such that
$v_i$ is the Gauss valuation on $F(X_L)=L(x_i)$ with respect to $x_i$.  The
coordinate $x_i$ is related to $x$ by a fractional linear transformation
\[
     x = \frac{a_ix_i+b_i}{c_ix_i+d_i},
\]
with $a_id_i-b_ic_i\neq 0$.  It can be shown that the model $\X$ is uniquely
determined by the set $\{v_1,\ldots,v_n\}$, see \cite{superell} or
\cite{JulianDiss}.

Let $\Y$ denote the normalization of $\X$ inside the function field
$F(Y_L)$. Then $\Y$ is a normal integral model of $Y_L$. In general, $\Y$ has
no reason to be semistable, and it is not clear in general how to describe its
special fiber $\Yb:=\Y\otimes k$. However, each irreducible component
$W\subset\Yb$ corresponds again to a discrete valuation $w$ on $F(Y_L)$
extending $v_L$, such that the residue field of $w$ is the function field of
$W$. It can be shown that this gives a bijection between the irreducible
components of $\Yb$ and the set of discrete valuations on $F(Y_L)$ extending
one of the valuations $v_i$ (see e.g.\ \cite{JulianDiss}, \S~3). In many
situations, the knowledge of all extensions of the $v_i$ to $F(Y_L)$ will give
enough information to decide whether the model $\Y$ is semistable and to
describe its special fiber.

We need one more piece of notation. For $m>1$ prime to $3$ we set
\[
      L_m:=K(\pi)/K
\]
where $\pi^m=-3$. Then $L_m/K$ is a tamely ramified Galois extension of degree
$m$. The Galois group $\Gamma:={\rm Gal}(L_m/K)$ is cyclic and generated by
the element $\tau\in\Gamma_m:=\Gal(L_m/K)$ determined by
\[
    \tau(\pi)=\zeta_m\pi,
\]
where $\zeta_m\in K$ is a primitive $m$th root of unity (which exists because
$k$ is algebraically closed). Note also that $L_m$ contains the third root of
unity
\[
      \zeta_3 := \frac{-1+\pi^{m/2}}{2}.
\]
We remark that the choice of $\tau$ and $m$ agrees with the notation chosen in
\S~\ref{sec:picard_wild2}

\begin{example}\label{exa:potgood}

Let $Y$ be the Picard curve over $K$ given by the equation
\begin{equation} \label{eq:exa1.1}
     y^3 = x^4+1.
\end{equation}
We claim that $Y$ has potentially good reduction, which is attained
over the tame extension $L:=L_8=K(\pi)/K$, with $\pi^8=-3$. 

To prove this, we apply the coordinate changes
\[
       x = \pi^3x_1, \quad y = 1 + \pi^4 y_1
\]
to \eqref{eq:exa1.1}. After a brief calculation, we obtain the new equation
\begin{equation} \label{eq:exa1.2}
   y_1^3 - \pi^4 y_1^2 - y_1 = x_1^4. 
\end{equation}
Equation \eqref{eq:exa1.2} is equivalent to \eqref{eq:exa1.1} in the sense that
it defines a curve over $K$ which is isomorphic to $Y$. Also,
\eqref{eq:exa1.2} defines an integral model $\Y$ of $Y_L$. Its special fiber
is the curve over $k=\bar{\FF}_3$ given by the (affine) equation
\[
   \Yb:\; y_1^3 - y_1 = x_1^4.
\]
This is a smooth curve of genus $3$. It follows that $\Y$ has good reduction
over $L$, as claimed.

Since $Y$ acquires stable reduction over a tame extension $L/K$, Lemma
\ref{lem:delta} implies that $f_3=\epsilon$. Equations (\ref{eq:fp2})
and (\ref{eq:main2}) imply that
$f_3=6$.  

For completeness, we compute the action of $\Gamma^0=\gen{\tau}$ on
$\Yb$ explicitly. We consider $\tau$ as an automorphism of the
structure sheaf of $\Y$. By definition, we have
\[
   \tau(\pi) =\zeta_8\pi, \quad \tau(x) = x,\quad \tau(y) = y.
\]
It follows that
\[
    \tau(x_1) =\zeta_8^5x_1, \quad \tau(y_1) = - y_1.
\]
This describes $\tau|_{\Yb}$ as an automorphism of $\Yb$ of order
$8$, as expected from the proof of Lemma \ref{lem:genuscomp}. 
\end{example}

\begin{example} \label{exa:f_3=4I}
  Let $Y/K$ be the Picard curve
  \begin{equation} \label{eq:exa2.1}
      Y:\; y^3 = f(x) := 3x^4+x^3-54.
  \end{equation}
  We claim that $Y$ has semistable reduction over the tame extension
  $L:=L_4/K$. Moreover, the stable reduction $\Yb$ is as in Case (b) of
  Theorem \ref{thm:wild1}, and $f_3=4$. 

  First we define a semistable model $\X$ of $X_L:=\PP^1_L$ by specifying two
  discrete valuations $v_1,v_2$ on $F(X_L)$ which extend $v_L$. We  then
  show that the normalization $\Y$ of $\X$ in $F(Y_L)$ is the stable model of
  $Y_L$, and  determine its special fiber $\Yb$ and the action of the
  inertia group of $L/K$ on $\Yb$.  

  The valuation $v_1$ is defined as the Gauss valuation on $F(X_L)=F(x_1)$
  with respect to the coordinate $x_1$, which is related to $x$ by 
  \begin{equation} \label{eq:exa2.2}
        x = \pi^2 x_1.
  \end{equation}
  We claim that $v_1$ has a unique extension $w_1$ to $F(Y_L)$ that is
  unramified. To show this, we need a so-called {\em $p$-approximation} of $f$
  with respect to $v_1$, see \cite{superp}. In fact, we can write 
  \[
      f = \pi^6\big(x_1^3 +\pi^6(2-x_1^4)\big).
  \]
  Here we have used the relation $\pi^4=-3$. This suggests the coordinate
  change
  \begin{equation} \label{eq:exa2.2a}
      y = \pi^2(x_1 + \pi^2 y_1).
  \end{equation}
  After a short calculation we obtain a new equation for $Y_L$:
  \begin{equation} \label{eq:exa2.3}
    y_1^3 - \pi^2x_1y_1^2 - x_1^2y_1 = 2-x_1^4.
  \end{equation}
  If we consider \eqref{eq:exa2.3} as defining an affine curve over $\OO_L$,
  its special fiber is the affine curve over $k$ with equation
  \begin{equation} \label{eq:exa2.4}
      \yb_1^3-\xb_1^2\yb_1 = -1-\xb_1^4.
  \end{equation}
  In fact, \eqref{eq:exa2.4} defines an irreducible affine curve with a cusp
  singularity in $(\xb_1,\yb_1)=(0,-1)$. It follows that the inverse image in
  $\Yb$ of $\Xb_1$ is an irreducible component $W_1$ of multiplicity one
  birationally equivalent to the curve given by \eqref{eq:exa2.4}. To compute
  the geometric genus of $W_1$ we substitute $\yb_1=-1+\xb_1\zb_1$ into
  \eqref{eq:exa2.4} and obtain the Artin--Schreier equation
  \begin{equation} \label{eq:exa2.4b}
    \zb_1^3-\zb_1 = -\xb_1^{-1}-\xb_1.
  \end{equation}
  Using the Riemann--Hurwitz formula, one sees that $W_1$ has geometric genus
  $2$. 

  The valuation $v_2$ of $F(X_L)$ corresponds to the choice of the coordinate
  $x_2$ given by
  \begin{equation} \label{eq:exa2.5}
      x = 3(1 + \pi x_2).
  \end{equation}
  After a short calculation we can write
  \begin{equation} \label{eq:exa2.6}
       f = 3^3\big(\,(-1+\pi x_2)^3 - 2\pi^6x_2^2 + 3^2(\ldots)\,\big).
  \end{equation}
  This suggests the change of coordinate
  \begin{equation}  \label{eq:exa2.7}
      y = 3\big(\,(-1+\pi x_2) + \pi^2 y_2\,\big).
  \end{equation}
  Plugging in \eqref{eq:exa2.7} into \eqref{eq:exa2.1} and using
  \eqref{eq:exa2.6} we arrive at the equation
  \begin{equation} \label{eq:exa2.8}
     y_2^3 +\pi^2(-1+\pi x_2)y_2^2 - (-1+\pi x_2)^2y_2 = -2 x_2^2 +\pi^2(\ldots).
  \end{equation}
  Reducing \eqref{eq:exa2.8} modulo $\pi$ we obtain the irreducible equation
  \begin{equation} 
     \yb_2^3 -\yb_2 = \xb_2^2,
  \end{equation}
  which defines a curve of genus $1$. It follows that the inverse image of
  $\Xb_2$ in $\Yb$ is an irreducible projective curve $W_2$ of geometric genus
  $1$. 

  So $\Yb$ consists of two irreducible components $W_1$ and $W_2$ of
  geometric genus $2$ and $1$. On the other hand, $\Y_s$ is known to have
  arithmetic genus $3$. By a standard argument (see e.g.\ ) we can conclude
  that $W_1$, $W_2$ are smooth and meet transversely in a single point. This
  shows that $Y$ has semistable reduction over the tame extension $L_4/K$,
  with a stable model of type (b).

  \bigskip Let us try to analyze the action of $\Gamma=\Gamma^0=\gen{\tau}$ on
  $\Yb$. By definition, $\tau(\pi) =\zeta_4\pi$, $\tau(x)=x$ and
  $\tau(y)=y$. 
  From \eqref{eq:exa2.2} and \eqref{eq:exa2.2a} we deduce that
  $\tau|_{W_1}$ is given by
  \[
      \tau(\xb_1) = -\xb_1, \quad \tau(\yb_1) = \yb_1,\quad
      \tau(\zb_1)=-\zb_1.
  \]
  From \eqref{eq:exa2.5} and \eqref{eq:exa2.7} we see that
  \[
     \tau(\xb_2) = \zeta_4^3\xb_2, \quad  \tau(\yb_2) = -\yb_2.
  \]
  It follows that the curve $\Yb^0:=\Yb/\Gamma^0$ has two irreducible smooth
  components, $W_1^0=W_1/\Gamma^0$ and $W_2^0=W_2/\Gamma^0$, meeting in a
  single point. An easy calculation (compare with the proof of Lemma
  \ref{lem:genuscomp}) shows that $g(W_1^0)=1$ and $g(W_2^0)=0$. It follows
  that $g(\Yb^0)=1$ and $\dim\, H^1(\Delta_{\Yb^0})=0$ and hence $f_3 = 6- 2 =
  4$.  
\end{example}

\begin{remark}
  The two examples discussed above are quite special. Typically, the extension
  $L/K$ needs to be wildly ramified, and have rather large degree. It is then
  hard (and often practically impossible) to do computations as above by
  hand. Most of the examples  in \cite{MichelDiss} and this paper have
  been computed with the help of (earlier versions of) Julian R\"uth's Sage
  packages {\tt mac lane} and {\tt completion} (available at
  \url{https://github.com/saraedum}), and the algorithms from 
  \cite{superell} and \cite{superp}.
\end{remark}


\section{The tame case: $p\neq 3$}

\label{sec:tame}

In this section we assume that the residue characteristic $p$ of our ground
field $K$ is different from $3$. In this case it is much easier to analyze the
semistable reduction of Picard curves and to compute the conductor exponent
$f_\p$ than for $p=3$. The theoretical background for this are the {\em
  admissible covers}, see \cite{HarrisMumford82}, \cite{liu}, \S~10.4.3, or
\cite{tame}. In the case of superelliptic curves the computation of $f_\p$ has
already been described in detail in \cite{superell}, hence we can be much
briefer than in the previous section.

\subsection{The stable model}

Let $K$ be as in \S~\ref{setup}, with $p\neq 3$. Let
$Y/K$ be a Picard curve, given by an equation
\[
    Y:\; y^3 = f(x),
\]
where $f\in K[x]$ is a separable polynomial of degree $4$. Let $L_0/K$
denote the splitting field of $f$. Let $L/L_0$ be a finite extension
with ramification index $3$ such that $L/K$ is a Galois
extension. Then \cite{superell}, Corollary 4.6 implies that $Y$
acquires semistable reduction over $L$.

We note in passing that $L/K$ is tamely ramified unless $p=2$. This follows
from the definition of the Galois extension $L_0/K$, whose degree divides
$4!=24$.  

A semistable model $\Y$ of $Y_L$ may be constructed as follows, see
\cite{superell}, \S~4. Let $D\subset X=\PP_K^1$ denote the branch divisor of the
cover $\phi:Y\to X$, consisting of the set of zeros of $f$ and
$\infty$. Since $L$ contains the splitting field of $f$, the pullback
$D_L\subset Y_L$ consists of $5$ distinct $L$-rational points. Let $(\X,\D)$
denote the {\em stably marked model} of $(X_L,D_L)$. By this we mean that $\X$
is the minimal semistable model of $X_L$ with the property that the schematic
closure $\D\subset\X$ of $D_L$ is \'etale over $\Spec\OO_L$ and contained
inside the smooth locus of $\X\to\Spec\OO_L$. Let $\Xb:=\X\otimes\FF_L$ denote
the special fiber of $\X$ and $\Db=\D\cap\Xb$ the specialization of
$D_L$. Then $(\Xb,\Db)$ is a stable $5$-marked curve of genus zero. This means
that $\Xb$ is a tree of projective lines, where every irreducible component
has at least three points which are either marked (i.e. lie in the support of
$\Db$) or are singular points of $\Xb$. 

Let $\Y$ denote the normalization of $\X$ with respect to the cover $Y_L\to
X_L$. Theorem 3.4 from \cite{superell} shows that $\Y$ is a quasi-stable model
of $Y_L$. A priory, it is not clear whether $\Y$ is the stable model of
$Y$. The following case-by-case analysis will show that it is. 

We will use the fact that the natural map $\Y\to\X$ is an {\em admissible cover}
with branch locus $\D$. In particular, the induced map
\[
   \phib:\Yb\to\Xb
\]
between the special fiber of $\Y$ and of $\X$ is generically \'etale
and identifies $\Xb$ with the quotient scheme $\Yb/G$.  

We describe the restriction of the map $\phib$ to an irreducible
component $\Xb_i$ of $\Xb$.  Without loss of generality we may assume
that $K$ (and hence $L$) contains a primitive $3$rd root of unity
$\zeta_3$, which we fix. For each branch point $\xi$ of
$\phib|_{\Xb_i}$ the {\em canonical generator of inertia} $g\in G$ is
characterized by $g^\ast u\equiv \zeta_3 u\pmod{u^2}$, where $u$ is a
local parameter at $\phib|_{\Xb_i}^{-1}(\xi_i)$. A branch point of
$\phib|_{\Xb_i}$ is either the specialization of a branch point of
$\phi$ or a singular point of $\Xb$. 

Assume that $\xi$ is the specialization of a branch point. An
elementary calculation shows that the canonical generator of inertia is
$\sigma$ of $\xi$ is the specialization of $\infty$ and $\sigma^2$
otherwise. Now let $\xi$ be a singularity of $\Xb$, and denote the
irreducible components intersecting in $\xi$ by $\Xb_1$ and $\Xb_2$.
Then the canonical generators $g_i$ of the restrictions
$\phib|_{\Xb_i}$ at $\xi$ satisfy
\[
g_1=g_2^{-1}.
\]
(This last condition says that $\phib$ is an admissible cover.)

The upshot is that the map $\phib:\Yb\to\Xb$ is completely determined
and easily described by the stably marked curve $(\Xb,\Db)$.

The following lemma lists the 5 possibilities for $\Xb$.  Note that we
need to distinguish between $\infty$ and the other $4$ branch
points. The proof is elementary, and therefore omitted.

\begin{lemma}\label{lem:tame}
With assumptions and notations as in the beginning of the section, we
have the following $5$ possibilities for $\Xb$.
\begin{enumerate}[(a)]
\item The curve $\Xb$ is irreducible.
\item The curve $\Xb$ consists of two irreducible components
  $\Xb_1$ and $\Xb_2$. Three of the branch points of $\phi$ including
  $\infty$ specialize to $\Xb_1$, the other two to $\Xb_2$.
\item The curve $\Xb$ consists of three irreducible components
  $\Xb_1$, $\Xb_2$, and $\Xb_3$, where $\Xb_1$ and $\Xb_3$ intersect
  $\Xb_2$. The branch point $\infty$ specializes to $\Xb_2$, two other
  branch points specialize to $\Xb_1$, and two to $\Xb_3$.
\item  The curve $\Xb$ consists of two irreducible components
  $\Xb_1$ and $\Xb_2$. Three of the branch points of $\phi$ different from
  $\infty$ specialize to $\Xb_1$, the other two to $\Xb_2$.
\item The curve $\Xb$ consists of three irreducible components
  $\Xb_1$, $\Xb_2$, and $\Xb_3$, where $\Xb_1$ and $\Xb_3$ intersect
  $\Xb_2$. Two branch points including $\infty$ specialize to $\Xb_1$,
  two other branch points specialize to $\Xb_3$, and the last one to
  $\Xb_2$.
\end{enumerate}
\end{lemma}

The following result immediately follows from the possibilities for
$\Xb$, together with the fact that $\phib$ is an admissible cover.

\begin{theorem} \label{thm:tame1}
  Let $K$ be as in \S\ref{setup}, with $p\neq 3$. Let $Y$ be
  a Picard curve over $K$, $L/K$ a finite Galois extension over which $Y$ has
  semistable reduction. Let $\Y$ denote the stable model of $Y_L$ over $\OO_L$
  and $\Yb:=\Y\otimes k$ the special fiber. Then $\Yb$ is as in one of the
  following five cases.
  \begin{enumerate}[(a)]
  \item
    The curve $\Yb$ is smooth.
  \item
    The curve $\Yb$ consists of two irreducible components, of genus $2$ and
    $1$, which intersect in a unique singular point.  
  \item The curve $\Yb$ has three irreducible components $W_1,W_2,W_3$ which
    are each smooth of genus $1$. There are two singular points where $W_1$
    (resp.\ $W_3$) intersects $W_2$. 
  \item
    There are two irreducible components $W_1,W_2$ of genus $0$ and
    $1$, respectively, and three singular points where $W_1$ and $W_2$
    intersect.
  \item There are three irreducible components $W_1,W_2,W_3$, of genus
    $0$, $0$ and $1$, respectively, and $4$ singular points. The
    components $W_1$, $W_2$ meet in three of these singular points,
    while $W_2$ and $W_3$ meet in the fourth.
  \end{enumerate}
\end{theorem} 


\subsection{The conductor exponent in the tame case}

In the tame case, there are no useful lower bounds for the conductor
exponent. In particular, $Y$ may have good reduction in which case we have
$f_\p=0$. Also, unlike for $p=3$, nothing is gained by assuming that the
ground field $K$ is totally unramified. Still, some useful restrictions on
$f_\p$ can be proved (see Theorem \ref{thm:tame2} below). 

We start by recalling a well known criterion for good reduction, see e.g.\
\cite{Holzapfel}, \S~7. Let 
\[
   Y:\ y^3 = f(x) = a_4x^4+a_3x^3+a_2x^2+a_1x +a_0
\]
be a Picard curve over $K$.  Replacing $(x,y)$ by
$(a_4^{-1}x,a_4^{-1})$ and multiplying both sides of the defining
equation by $a_4^3$, we may assume that $a_4=1$. Let $\Delta(f)\in
K^\times$ denote the discriminant of $f$.  (Since we assume that $f$ is
separable, we have $\Delta(f)\neq 0$.) After replacing $(x,y)$ by $(u^{-3}x,
u^{-4}y)$ and multiplying by $u^{12}$ on both sides, for a suitable
$u\in K^\times$, we may further assume that all coefficients $a_i\in\OO_K$ are
integral. In particular, it follows that
$\Delta(f)\in\OO_K$. Since
\[
   \Delta(u^{12}f(u^{-3}x)) = u^{36}\Delta(f),
\]
 by the right choice of $u$, we may assume that 
\begin{equation}
       0\leq \ord_\p(\Delta(f)) < 36.
\end{equation}

\begin{lemma}\label{lem:goodred}
   Assume that the Picard curve $Y$ is given by a minimal equation over
   $\OO_K$, as above. Then $Y$ has good reduction if and only if
   $\Delta(f)\in\OO_K^\times$. 
\end{lemma}

\begin{proof}
See \cite{Holzapfel}, Lemma 7.13.
\qed
\end{proof}

Note that the forwards direction of Lemma \ref{lem:goodred} also
follows from Theorem \ref{thm:tame1}.
Here is what we can say in general about the conductor exponent.

\begin{theorem} \label{thm:tame2}
  Let $K$ be as before, with $p\neq 3$, and $Y$ a Picard
  curve over $K$. Let $f_\p$ denote the conductor exponent for $Y$, relative
  to the prime ideal $\p$ of $\OO_K$. Then the following holds.
  \begin{enumerate}[(a)]
  \item
    If $f_\p=0$ then the stable reduction of $Y$ is as in Case (a), (b), or (c)
    of Theorem \ref{thm:tame1}. Furthermore, the splitting field $L_0/K$ of
    $f$ is unramified at $\p$.
\item If $p=2$ then $f_\p\neq 1$. 
  \item
    If $p\geq 5$ then $f_\p\in \{0,2,4,6\}$.
 \end{enumerate}
\end{theorem}

\begin{proof} 
We start be proving Statement (a). Note that $f_\p=0$ if and only if
$\delta=0$ and $\dim\, H^1_{\et}(\Yb^0, \QQ_\ell)=6$.  The second
condition, together with the discussion after Proposition
\ref{prop:fp}, implies that $\gamma(\Yb^0)$. Statement (a) now follows
immediately from Theorem \ref{thm:tame2}.

{\bf Claim}: The integer $\epsilon$, defined in
Proposition \ref{prop:fp},  is even.  
The discussion following Proposition \ref{prop:fp}  implies that
$f_\p$ is odd if and only if $\dim\, H^1(\Delta_{\Yb^0}) $ is odd. The
case distinction in Theorem \ref{thm:tame1} implies that $\dim\,
H^1(\Delta_{\Yb^0}) $ is at most $2$. Therefore to prove the claim, it
suffices to show that $\gamma(\Yb^0)=\dim\, H^1(\Delta_{\Yb^0})\neq
1.$ We prove this in the case that $\Yb$ is as in (d) of Theorem
\ref{thm:tame1}. The argument in the case that $\Yb$ is as in (e) is
very similar. In the other cases there is nothing to prove.

Assume that $\Yb$ is as in (d) of Theorem \ref{thm:tame1}. Then $\Xb$
is as (d) of Lemma \ref{lem:tame} and $\phib$ maps $W_i$ to $\Xb_i$.
Since $\infty$ is $K$-rational, the monodromy group $\Gamma$ fixes
it. It follows that $\Gamma$ acts on the component $\Xb_2$ to which
$\infty$ specializes. (This is similar to the argument in the proof of
\cite{superell}, Lemma 5.4.) Since there is exactly one other branch
point specializing to $\Xb_2$, this point is fixed by $\Gamma$, as
well. Similarly, $\Gamma$ fixes the unique singularity. Since $\Gamma$
fixes at least $3$ points on the genus-$0$ curve $\Xb_2$, it acts
trivially on $\Xb_2$. Equation (\ref{eq:picard_wild2}) implies that
the action of $\Gamma$ on $\Yb$ descends to $\Xb$. It follows that
$\Gamma$ acts on $W_2$ via a subgroup of $G$. We conclude that
$\Gamma$ either fixes the three singularities of $\Yb$ or cyclically
permutes them. It follows that $\gamma({\Yb^0})$ is $2$ or
$0$. This proves the claim.

Assume that $p=2$. Using Equation (\ref{eq:fp3}) one shows that if
$\delta\neq 0$ then $\delta\geq 2$. Therefore Statement (b) follows
from the claim.

For Statement (c) recall that $L/K$ is at most tamely ramified for $p\geq
5$. It follows that $\delta=0$, and hence that $f_\p=\epsilon$
is bounded by $2g(Y)=6$. Statement (c) now follows from the claim.
\qed
\end{proof}

\begin{remark}\label{rem:delta} 
\begin{enumerate}[(a)]
\item The condition $f_\p=0$ in Theorem \ref{thm:tame2}.(a) is
  equivalent to the condition that the Jacobian variety of $Y$ has
  good reduction over $K$. This is the case if and only if $Y$ has
  stable reduction already over $K$, and the graph of components
  $\Delta_{\Yb}$ is a tree.  This observation is similar to the
  statement of Lemma \ref{lem:delta}.
\item For $p=2$ the conductor exponent $f_2$ may be odd. An example
  can be found in Example \ref{exa:f2odd}.
\item The bound on $f_\p$ for $p= 5,7$ in Theorem \ref{thm:tame2}.(c)
  is slightly sharper than the bound for $f_\p$ for general abelian
  varieties of dimension $3$ from \cite{BrumerKramer}, Thm.~6.2. The
  reason is that Brumer and Kramer obtain an upper bound for $\delta$.
 For Picard curves and $p=5,7$ we have $\delta=0$, whereas this is not
  necessarily the case for general curves of genus $3$.

  For $p=2$ the result of \cite{BrumerKramer} yields the upper bound $f_\p\leq
  28$. Distinguishing the possibilities for the stable reduction and combining
  our arguments with those of \cite{BrumerKramer} it might be possible to
  improve the bound in this case.
\end{enumerate}
\end{remark}

\begin{example} \label{exa:tame}
Consider the Picard curve
\[
  Y:\; y^3 = f(x) = x^4 + 14x^2+72x-41
\]
over $K:=\QQ_5^{\rm nr}$. We claim that $Y$ has semistable reduction over $K$,
and that the reduction type is as in Case (b) of Theorem
\ref{thm:tame1}. Therefore, $f_5=0$. 

We will argue in a similar way as in \S~\ref{subsec:wild_examples},
see in particular Example \ref{exa:f_3=4I}, see also \cite{superell},
\S~6 and \S~7. The first observation is that
\begin{equation} \label{eq:wildexa1}
    f = x^4+14x^2+72x-41 \equiv (x+3)^2(x^2+4x+1) \pmod{5}.
\end{equation}
By Hensel's Lemma, $f$ has two distinct roots
$\alpha_1,\alpha_2\in\OO_K$ with $\alpha_i^2+4\alpha_i+1\equiv
0\pmod{5}$.  The other two roots of $f$ are congruent to $-3\pmod{5}$.
Substituting $x=-58+5^3x_1$ into $f$, we see that
\begin{equation} \label{eq:wildexa2}
    f \equiv 5^6(3x_1^2+4x_1+2) \pmod{5^7}.
\end{equation}
It follows that $f$ has two more roots $\alpha_3,\alpha_4\in K$ of the form
$\alpha_i=-58+5^3\beta_i$, with $\beta_i\in\OO_K$ and
$3\beta_i^2+4\beta_i+2\equiv 0\pmod{5}$. So $f$ splits over $K$. 

Let $(\X,\D)$ be the stably marked model of $(X,D)$, where $X=\PP^1_K$
and $D=\{\infty, \alpha_1,\ldots,\alpha_4\}$. The calculation of the
$\alpha_i$ above show that $\X$ is the $\OO_K$-model of $X$
corresponding to the set of valuations $\{v_0,v_1\}$, where $v_0$
(resp.\ $v_1$) is the Gauss valuation on $K(x)$ with respect to the
parameter $x$ (resp.\ to $x_1$). Let $\Y$ be the normalization of $\X$
in the function field of $Y$. We claim that the special fiber $\Yb$ of $\Y$
consists of two irreducible components $W_0,W_1$ of geometric genus
$2$ and $1$, respectively. By the same argument as in Example
\ref{exa:f_3=4I}, this already implies that $\Y$ is semistable and
that the special fiber is as in Case (b) of Theorem \ref{thm:tame1}.

To prove the claim it suffices to find generic equations for $W_0$ and
$W_01$. For $W_0$ we just have to reduce the original equation for $Y$ modulo
$5$. By \eqref{eq:wildexa1} we obtain 
\[
     W_0:\; \yb^3 = (\xb+3)^2(\xb^2+4\xb+1),
\]
which shows that $g(W_0)=2$. 
For $W_1$ we write $f$ as a polynomial in $x_1$, substitute $y=5^2w$, divide
by $5^6$ and reduce modulo $5$. By \eqref{eq:wildexa2} we obtain
\[
   W_1:\; \bar{w}^3 = 3\xb_1^2+4\xb_1+2,
\]
which shows that $g(W_1)=1$. Now everything is proved.  
\qed
\end{example}

\begin{remark}
  The example above is again rather special, since $f_5=0$ even though $Y$ has
  bad reduction at $p=5$. (See also Definition \ref{def:excprime}).
\end{remark}

\section{Searching for Picard curves over $\QQ$ with small conductor} \label{sec:search}

In this last section we briefly address the problem of constructing Picard
curves with small conductor. We think this is an interesting problem which
deserves further investigation. The main background result here is the {\em
  Shafarevic conjecture} (which is a theorem due to Faltings). We use this
theorem via the following corollary.

\begin{theorem}[Faltings] \label{thm:Faltings}
  Fix a number field $K$ and an integer $g\geq 2$.
  \begin{enumerate}[(a)]
  \item For any finite set $S$ of finite places of $K$ there exist at most a
    finite number of isomorphism classes of smooth projective curves of genus
    $g$ over $K$ with good reduction outside $S$.
  \item
    For any constant $N>0$ there exists at most a finite number of isomorphism
    classes of curves of genus $g$ over $K$ with conductor $\leq N$.
  \end{enumerate} 
\end{theorem}

\begin{proof}
  Satz 6 in \cite{Faltings83} states that there are at most a finite number
  of $d$-polarized abelian varieties of dimension $g$ over $K$ with good
  reduction outside $S$, for fixed $K$, $g$, $d$ and $S$. Statements (a) and
  (b) follow from this. For (a), one simply uses Torelli's theorem (see
  \cite{Faltings83}, p. 365, Korollar 1). To deduce (b) we use that the
  conductor of a curve $Y$ is the same as the conductor of its Jacobian, and
  that an abelian variety over $K$ has bad reduction at a finite place $\p$ of
  $K$ if and only if $f_\p=0$ (see e.g.\ \cite{SerreTate68}, Theorem 1).
\qed
\end{proof}

Unfortunately, no effective proof of Theorem \ref{thm:Faltings} is known in
general.\footnote{The precise meaning of an {\em effective proof} is that it
  provides an explicitly computable bound on the height of the curve or
  abelian variety in question.} However, for some special classes of curves
effective proofs are known, see e.g.\ \cite{vonKaenel14}.

The problem we wish to discuss here is whether the statement of Theorem
\ref{thm:Faltings} can be made computable in the case of Picard curves. More
precisely: given a finite set $S$ of rational primes (or a bound $N>0$), can
we compute the finite set of curves with good reduction outside $S$ (resp.\
with conductor $\leq N$)?  Note that this is not equivalent to (and may be
much easier than) having an effective proof of Theorem \ref{thm:Faltings} for
Picard curves. For the first problem, the answer is known to be affirmative:
 
\begin{proposition} \label{prop:S-unit}
  There exists an algorithm which, given as input a number field $K$ and
  finite set $S$ of finite places of $K$, computes the set of isomorphism
  classes of all Picard curves $Y/K$ with good reduction outside $S$.
\end{proposition}

\begin{proof}
  This is an adaption to Picard curves of the algorithm given by Smart for
  hyperelliptic curves, see \cite{Smart97} and \cite{MalmskogRasmussen}. The
  idea is that it suffices to determine the finite set of equivalence classes
  of binary forms of degree $4$ over $K$ whose discriminant is an $S$-unit
  (corresponding to the polynomial $f(x)$). The latter problem can be reduced
  to solving an $S$-unit equation, for which effective algorithms are known.
\qed
\end{proof}

\begin{example}
  Let $K=\QQ$ and $S=\{3\}$. Then there are precisely $63$ isomorphism classes
  of Picard curves over $\QQ$ with good reduction outside $S$. See
  \cite{MalmskogRasmussen}. 

  For example, the curve 
  \[
      Y:\; y^3 = f(x) = x^4-3x^3-24x^2-x
  \]
  has good reduction outside $S=\{3\}$ (the discriminant of $f$ is
  $\Delta(f)=3^{10}$). The stable reduction $\Yb$ of $Y$ at $p=3$ is
  as in Case (c) of Theorem \ref{thm:wild1}, the exponent conductor is
  $f_3=10$ (see \cite{MichelDiss}, Appendix A1.1).  This is the lowest
  value for the conductor which occurs for the curves in the list of
  \cite{MalmskogRasmussen}. The conductor exponents of all $63$ Picard
  curves from \cite{MalmskogRasmussen} have been computed in
  \cite{MichelDiss}, Appendix A1.2. From this calculation it follows
  that the conductor exponent $f_3$ only takes the values $f_3=10, 11,
  12, 13, 15, 17, 19, 21$.

The upper bound on the conductor exponent from abelian varieties of
genus $3$ from \cite{BrumerKramer}, Theorem 6.2 yields $f_3\leq
21$. The result stated above therefore implies that this bound is also
obtained for Picard curves.
\end{example}

Unfortunately we do not know any algorithm for solving (b), i.e.\ for finding
all Picard curves with bounded conductor. The reason that the method for (a)
does not solve (b) is the existence of {\em exceptional primes}.

\begin{definition}\label{def:excprime}
  Let $Y$ be a Picard curve over $\QQ$ and $p$ a prime number. Then $p$ is
  called {\em exceptional} with respect to $Y$ if $Y$ has bad reduction at $p$
  and $f_p=0$.
\end{definition}

Exceptional primes are rather rare. It can easily be shown, using the
arguments from this paper, that if $p$ is a
exceptional prime for $Y$ then the splitting field of the polynomial $f$ is
unramified at $p$, and 
\[
    {\rm ord}_p(\Delta(f)) \in \{ 6,12\}.
\]

\begin{example}\label{exa:f2odd}
  We consider the Picard curve over $\QQ$
  \[
     Y:\; y^3 = f(x) = x^4 + 14x^2 + 72x -41.
  \]
  The discriminant of $f$ is $\Delta(f) = -2^{10}3^45^6$. So $Y$ has good
  reduction outside $S=\{2,3,5\}$. We have shown in Example \ref{exa:tame}
  that $f_5=0$, i.e.\ that $5$ is an exceptional prime. Using the methods of
  \cite{superell} and \cite{superp} one can prove that $f_2=19$ and $f_3=13$
  (see e.g.\ this SageMathCloud worksheet: \url{http://tinyurl.com/hp3qzmo},
  \cite{sage}). All in all, the conductor of $Y$ is
  \[
      N_Y = 2^{19}3^{13} = 835884417024.
  \]
  Although $S$ is small and $p=5$ is an exceptional prime, $N_Y$ is relatively
  large. We have tried but were not able to find a similar example with
  exceptional primes and a significantly smaller conductor. Nevertheless, the
  fact that exceptional primes exist means that we cannot easily bound the
  size of the set $S$ while searching for Picard curves with bounded
  conductor. 
\end{example}

Here is an example of a Picard curve with a relatively small conductor.

\begin{example} \label{exa:smallconductor}
  Consider the Picard curve 
  \[
      Y/\QQ:\; y^3 = f(x) = x^4 - 1.
  \]
  The discriminant of $f$ is $\Delta(f)=-256 = -2^8$. It follows that $Y$ has
  good reduction outside $S=\{2,3\}$. By \cite{MichelDiss}, \S~5.1.3, we have
  $f_2=6$ and $f_3=6$. Therefore,
  \[
      N_Y = 2^63^6 = 46656.
  \]
  
  The first named author has made an extensive search for Picard curves over
  $\QQ$ with small conductor (\cite{MichelDiss}, \S~5.3). Among all computed
  examples, the curve $Y$ was the one with the smallest conductor. 

A remarkable property of the curve $Y$ is that for every (rational)
prime $p$ it admits a map to $\PP^1$ of order prime to $p$, which
becomes Galois over an extension: besides the degree-$3$ map $\phi$
given by $(x,y)\mapsto x$, we have the map $(x,y)\mapsto y$, which has
degree $4$. In fact, the full automorphism group of $Y$ has order
$48$, and is maximal in the sense that $Y/\Aut_\CC(Y)$ is a projective
line, and the natural cover is branched at three points. 

  It is instructive to compare the above example with the curve
  \[
      Y':\; y^3 = x^4+1.
  \]
  This is a twist of $Y$. The curve $Y$ and $Y'$ become isomorphic over
  $\QQ[i]$, yet have different conductors. In fact, 
  \[
         N_{Y'} = 2^{16}3^6,
  \]
  see \cite{MichelDiss}, \S~5.1.2. 
\end{example}

We propose to study the following problem.

\begin{problem}
  Prove that the curve from Example \ref{exa:smallconductor} is the only
  Picard curve (up to isomorphism) with conductor $N_Y\leq 46656$, or find
  explicit counterexamples.
\end{problem}

Proposition \ref{prop:S-unit} and our main results (Theorem \ref{thm:main}
and Theorem \ref{thm:tame2}) suggest the following strategy for construction
Picard curves with small conductor and thereby finding counterexamples. If we
ignore the possibility of exceptional primes, a Picard curve with conductor
$\leq 2^63^6$ must have good reduction outside $S$, where $S$ is one of the
following sets:
\begin{itemize}
\item
  $\{2,3,p\}$, \quad $p\leq 13$,
\item
  $\{3,p\}$, \quad $p\leq 23$.
\end{itemize}
To find all such curves looks challenging but within reach. It should also be
very useful to take into account the local restrictions on the polynomial $f$
imposed by our results on curves with a specific value for $f_p$. On the other
hand, without an effective proof of Theorem \ref{thm:Faltings} (b) for Picard
curves, it is not clear at the moment how one could actually prove that the
curve from Example \ref{exa:smallconductor} (or any other curve we may find)
has minimal conductor.


\begin{thebibliography}{10}
\providecommand{\url}[1]{{#1}}
\providecommand{\urlprefix}{URL }
\expandafter\ifx\csname urlstyle\endcsname\relax
  \providecommand{\doi}[1]{DOI~\discretionary{}{}{}#1}\else
  \providecommand{\doi}{DOI~\discretionary{}{}{}\begingroup
  \urlstyle{rm}\Url}\fi

\bibitem{ICERM}
Bouw, I.I., Wewers, S.: {Semistable reduction of curves and computation of bad
  Euler factors of $L$-functions}.
\newblock
  \urlprefix\url{https://www.uni-ulm.de/fileadmin/website_uni_ulm/mawi.inst.100/mitarbeiter/wewers/course_notes.pdf}.
\newblock Notes for a minicourse at ICERM

\bibitem{superell}
Bouw, I.I., Wewers, S.: Computing {$L$}-functions and semistable reduction of
  superelliptic curves.
\newblock {G}lasgow {M}ath.~{J}. \textbf{59}, 77--108 (2017)

\bibitem{BrumerKramer}
Brumer, A., Kramer, K.: The conductor of an abelian variety.
\newblock Compositio Math. \textbf{92}(2), 227--248 (1994)

\bibitem{MichelDiss}
Börner, M.: {$L$-Functions of curves of genus $\geq 3$}.
\newblock Ph.D. thesis, Universität Ulm (2016).
\newblock \urlprefix\url{http://dx.doi.org/10.18725/OPARU-4137}

\bibitem{DeligneMumford69}
Deligne, P., Mumford, D.: The irreducibility of the space of curves of given
  genus.
\newblock Publ.\ Math.\ IHES \textbf{36}, 75--109 (1969)

\bibitem{Faltings83}
Faltings, G.: {Endlichkeitss\"atze f\"ur abelsche Variet\"aten \"uber
  Zahlk\"orpern}.
\newblock Inventiones Math. \textbf{73}, 349--366 (1983)

\bibitem{HarrisMumford82}
Harris, J., Mumford, D.: On the {Kodaira} dimension of the moduli space of
  curves.
\newblock Invent.\ Math. \textbf{67}, 23--86 (1982)

\bibitem{Holzapfel}
Holzapfel, R.P.: The Ball and Some Hilbert Problems.
\newblock Birkh\"auser (1995)

\bibitem{KoikeWeng}
Koike, K., Weng, A.: Construction of {CM} {P}icard curves.
\newblock Math. Comp. \textbf{74}(249), 499--518 (2005)

\bibitem{vonKaenel14}
von Känel, R.: {An effective proof of the hyperelliptic Shafarevich
  conjecture}.
\newblock Journal de Théorie des Nombres de Bordeaux \textbf{26}(2), 507--530
  (2014)

\bibitem{Liu94}
Liu, Q.: Conducteur et discriminant minimal de courbes de genre $2$.
\newblock Compositio Math. \textbf{94}(1), 51--79 (1994)

\bibitem{liu}
Liu, Q.: Algebraic geometry and arithmetic curves.
\newblock {Oxford University Press} (2006)

\bibitem{MalmskogRasmussen}
Malmskog, B., Rasmussen, C.: Picard curves over {${\mathbf{Q}}$} with good
  reduction away from $3$.
\newblock ArXiv:1407.7892

\bibitem{Picard}
Picard, E.: Sur des fonctions de deux variables indépendantes analogues aux
  fonctions modulaires.
\newblock Acta Math. \textbf{2}(1) (1883)

\bibitem{pries}
Pries, R.: Wildly ramified covers with large genus.
\newblock J. Number Theory \textbf{119}(2), 194--209 (2006)

\bibitem{Quine}
Quine, J.R.: Jacobian of the {P}icard curve.
\newblock In: Extremal {R}iemann surfaces ({S}an {F}rancisco, {CA}, 1995),
  \emph{Contemp. Math.}, vol. 201, pp. 33--41. Amer. Math. Soc., Providence, RI
  (1997)

\bibitem{Raynaud99}
Raynaud, M.: Sp{\'e}cialisation des rev{\^e}tements en caract{\'e}ristique $ p>
  0$.
\newblock Annales scientifiques de l'Ecole normale sup{\'e}rieure
  \textbf{32}(1), 87--126 (1999)

\bibitem{superp}
R\"uth, J., Wewers, S.: Semistable reduction of superelliptic curves of degree
  $p$.
\newblock In preparation

\bibitem{JulianDiss}
Rüth, J.: Models of curves and valuations.
\newblock Ph.D. thesis, Universität Ulm (2014).
\newblock \urlprefix\url{http://dx.doi.org/10.18725/OPARU-3275}

\bibitem{sage}
SageMath, I.: SageMathCloud Online Computational Mathematics (2016).
\newblock {\tt https://cloud.sagemath.com/}

\bibitem{SerreCL}
Serre, J.P.: Corps locaux.
\newblock Hermann, Paris (1968).
\newblock Troisi{\`e}me {\'e}dition, Publications de l'Universit{\'e} de
  Nancago, No. VIII

\bibitem{SerreZeta}
Serre, J.P.: Facteurs locaux des fonctions z\^eta des vari\'et\'es
  alg\'ebriques (d\'efinitions et conjectures).
\newblock S\'eminaire Delange-Pisot-Poitou (Th\'eorie des Nombres)
  \textbf{19}(2), 1--15 (1969)

\bibitem{SerreTate68}
Serre, J.P., Tate, J.: Good reduction of abelian varieties.
\newblock Annals of Math. \textbf{88}(3), 492--517 (1968)

\bibitem{SilvermanAT}
Silverman, J.H.: Advanced topics in the arithmetic of elliptic curves.
\newblock No. 151 in Graduate Text in Math. Springer-Verlag (1994)

\bibitem{Smart97}
Smart, N.P.: {$S$-unit equations, binary forms and curves of genus $2$}.
\newblock Proceedings of the London Mathematical Society \textbf{75}(2),
  271--307 (1997)

\bibitem{tame}
Wewers, S.: Deformation of tame admissible covers of curves.
\newblock In: H.~V\"olklein (ed.) Aspects of Galois theory, no. 256 in LMS
  Lecture Note Series, pp. 239--282 (1999)

\end{thebibliography}
\end{document}